\documentclass[12pt]{amsart}

\usepackage{comment}
\usepackage{color}
\usepackage{amsmath}
\usepackage{amsthm}
\usepackage{amssymb}
\usepackage{graphicx}
\usepackage{enumerate}
\usepackage{amsfonts}
\usepackage{mathrsfs}
\usepackage{parskip}
\usepackage{mathdots}
\usepackage{color}

\numberwithin{equation}{section}
\theoremstyle{plain}
\newtheorem{Proposition}[equation]{Proposition}
\newtheorem{Corollary}[equation]{Corollary}
\newtheorem*{Corollary*}{Corollary}
\newtheorem{Theorem}[equation]{Theorem}
\newtheorem*{Theorem*}{Theorem}
\newtheorem{Lemma}[equation]{Lemma}
\theoremstyle{definition}

\newtheorem{Conjecture}[equation]{Conjecture}
\newtheorem{Example}[equation]{Example}
\newtheorem{Remark}[equation]{Remark}

\allowdisplaybreaks

\def\C{\mathbb{C}}
\def\R{\mathbb{R}}
\def\D{\mathbb{D}}
\def\T{\mathbb{T}}

\def\K{\mathcal{K}}
\def\phi{\varphi}

\renewcommand{\ker}{\operatorname{Ker}}

\newcommand{\beqa}{\begin{eqnarray*}}
\newcommand{\eeqa}{\end{eqnarray*}}

\renewcommand{\leq}{\leqslant}
\renewcommand{\geq}{\geqslant}
\renewcommand{\subset}{\subseteq}

\author{Javad Mashreghi}
\address{D\'epartement de math\'ematiques et de statistique, Universit\'e Laval, Qu\'ebec, QC,
Canada, G1K 0A6}
\email{javad.mashreghi@mat.ulaval.ca}

\author[M. Ptak]{Marek Ptak}
\address{Department of Applied Mathematics,
University of Agriculture, ul. Balicka 253c\\ 30-198 Krak\'ow, Poland.}
\email{rmptak@cyf-kr.edu.pl}

\author[Ross]{William T. Ross}
	\address{Department of Mathematics and Computer Science, University of Richmond, Richmond, VA 23173, USA}
	\email{wross@richmond.edu}
	

\title[Outer functions]{Interpolating with outer functions }


\thanks{This work was supported by the NSERC Discovery Grant (Canada) and by the Ministry of Science and
Higher Education of the Republic of Poland.}

\begin{document}

\begin{abstract}
The classical theorems of Mittag-Leffler and Weierstrass show that when $\{\lambda_n\}$ is a sequence of distinct points in the open unit disk $\D$, with no accumulation points in $\D$,  and $\{w_n\}$ is any sequence of complex numbers, there is an analytic function $\phi$ on $\D$ for which $\phi(\lambda_n) = w_n$. 
A celebrated theorem of Carleson \cite{MR117349} characterizes when, for a bounded sequence $\{w_n\}$, this interpolating problem can be solved with a bounded analytic function. A theorem of Earl \cite{MR284588}  goes further and shows that when Carleson's condition is satisfied,  the interpolating function $\phi$ can be a constant multiple of a Blaschke product. In this paper, we explore when the interpolating $\phi$  can be an outer function. We then use our results to refine a result of McCarthy \cite{MR1065054} and explore the common range of the co-analytic Toeplitz operators on a model space. 
\end{abstract}

\maketitle

\section{Interpolation}

Interpolation problems for analytic functions have been a mainstay in complex analysis since its conception in the late 19th century. The general idea is that we have a certain class $\mathcal{X}$ of analytic functions on the open unit disk $\D$ (e.g., all analytic functions, bounded analytic functions, analytic self maps of $\D$, Blaschke products, outer functions). Then, for a sequence $\{\lambda_n\}$ of distinct points in $\D$  and sequence $\{w_n\}$ of complex numbers, we want to find an $f \in \mathcal{X}$ such that $f(\lambda_n) = w_n$ for all $n$. If we are not able to solve this problem for all $\{\lambda_n\}$ and $\{w_n\}$, what restrictions must we have? 

Suppose $\mathcal{X}$ is the class of {\em all} analytic functions on $\D$. For a sequence $\{\lambda_n\}$ of distinct points in $\D$ (with no limit point in $\D$) and any sequence $\{w_n\}$, an application of the classical Mittag--Leffler theorem and the Weierstrass factorization theorem produces an analytic function $f$ with $f(\lambda_n) = w_n$ for all $n$. In other words, for the class $\mathcal{X}$ of all analytic functions on $\D$,  besides the obvious restriction that $\{\lambda_n\}$ has no limit points in $\D$, there is no other restriction on $\{\lambda_n\}$ to be able to interpolate any sequence $\{w_n\}$ with an  analytic function. 

Of course, there are the finite interpolation problems. For example, a well known result of Lagrange, from 1795, says that given distinct $\lambda_1, \ldots, \lambda_n$ in $\C$ and arbitrary  $w_1, \ldots, w_n$  in $\C$ there is a polynomial $p$ of degree $n - 1$ such that $p(\lambda_j) = w_j$ for all $1 \leq j \leq n$. There is also the often-quoted result  of Nevanlinna and Pick (from 1916) which says that given distinct $\lambda_1, \ldots, \lambda_n$ in $\D$ and arbitrary $w_1, \ldots, w_n$ in $\D$, there is an $f \in H^{\infty}$ with $|f| \leq 1$ on $\D$ for which $f(\lambda_j) = w_j$, $1 \leq j \leq n$, if and only if the Nevanlinna-Pick matrix 
$$\Big[\frac{1 - \overline{w_j} w_i}{1 - \overline{\lambda_j} \lambda_i}\Big]_{1 \leq i, j \leq n}$$ is positive semidefinite \cite{MR1882259, Garnett}. 

When $\mathcal{X}$ is the class of bounded analytic functions on $\D$, denoted in the literature by $H^{\infty}$, a well-known theorem of Carleson \cite{MR117349} (see also \cite{Garnett})  says the following: A sequence $\Lambda = \{\lambda_n\} \subset \D$ has the property that given any bounded sequence $\{w_n\}$ there is a  $\phi \in H^{\infty}$ such that $\phi(\lambda_n) = w_n$ if and only if 
\begin{equation}\label{deltaaaa}
\delta(\Lambda) := \inf_{n \geq 1} \prod_{k=1, k \ne n}^{\infty} \left| \frac{\lambda_k-\lambda_n}{1-\overline{\lambda}_k \lambda_n}\right| > 0.
\end{equation}

Such $\{\lambda_n\}$ are called {\em interpolating sequences}. In this paper we explore the type of functions $\phi \in H^{\infty}$ that can perform the interpolating. For example, a result of Earl \cite{MR284588} says that when $\delta(\Lambda) > 0$ one can always take the interpolating function
$\phi$ to be a constant multiple of a Blaschke product. Other types of interpolation problems are discussed  in \cite{MR2309193, MR133446}. 

Inspired by a common range problem for co-analytic Toeplitz operators on model spaces that we will discuss at the end of this paper, we focus on conditions on the targets $\{w_n\}$ that allow us to take $\phi$ to be an outer function (bounded outer function). 
Our results are as follows: In Theorem \ref{IT001} we prove that for interpolating $\{\lambda_n\}$ and bounded $\{w_n\}$  with 
$$\inf_{n \geq 1} |w_n| > 0,$$  there is a bounded outer function $\phi$ such that $\phi(\lambda_n) = w_n$ for all $n$. As an application of this, we prove  in Proposition \ref{comp886bb} that when $\{w_n\}$ and $\{w^{'}_n\}$ are  bounded  with 
$$0 < m \leq \Big|\frac{w_n}{w^{'}_n}\Big| \leq M < \infty, \quad n \geq 1,$$ then $\{w_n\}$ can be interpolated by an outer function (bounded outer function)  if and only if $\{w^{'}_n\}$ can as well. Therefore, for subsequent discussions, without loss of generality, we may consider positive target sequences. To establish conditions for which $\{w_n\}$ can be interpolated by an outer function, we prove in Theorem \ref{T:interpol-outer-inner} that when $\{w_n\} \subset \C \setminus \{0\}$ can be interpolated by an outer function it must be the case that
$$\lim_{n \to \infty} (1 - |\lambda_n|) \log |w_n| = 0.$$ Hence sequences such as
$$w_n = e^{-\frac{1}{1-|\lambda_n|}}, \quad n \geq 1,$$
 can not be interpolated by an outer function.  In other words, any $H^{\infty}$ function $\phi$ for which $\phi(\lambda_n) = w_n$ for all $n$ (and such $\phi$ exist by Carleson's theorem) must have an inner factor.  In fact (Theorem \ref{Blakshkefactor}), any bounded analytic function $\phi$ which satisfies the stronger decay condition
$$\phi(\lambda_n) = e^{-\frac{1}{(1 - |\lambda_n|)^2}}$$  must have a Blaschke factor.   In Theorem \ref{growthragfetd6dx6} and Theorem \ref{MTThree} we discuss the sharpness of Theorem \ref{T:interpol-outer-inner} and explore conditions on the decay the rate of  $(1 - |\lambda_n|) \log |w_n|$  to determine when there is an outer function (bounded outer function) that interpolates $\{w_n\}$.  

Worth mentioning here is the paper \cite{MR2309193} which examines the question of when the interpolating function can be zero free. The outer functions are a strict subclass of the zero free functions since the zero free functions can have a singular inner factor. 

In the final part of this paper we apply our results to determine the common range for the co-analytic Toeplitz operators on a model space. In fact, this was our original reason for exploring this topic. For an inner function $u$, define the model space $\K_u = (u H^2)^{\perp}$. It is known \cite[p.~106]{MR3526203} that $\K_u$ is an invariant subspace for any co-analytic Toeplitz operator $T_{\overline{\phi}}$ on $H^2$.  In \cite{MR1065054} McCarthy described the set 
$$\mathscr{R}(H^2) := \bigcap \big\{T_{\overline{\phi}} H^2: \phi \in H^{\infty} \setminus \{0\}\big\},$$
the functions in the common range of all the (nonzero) co-analytic Toeplitz operators. By the Douglas factorization theorem (see \S \ref{Crrrsdsdsd546u} below), $H^{\infty} \setminus \{0\}$ can be replaced by $H^{\infty} \cap \mathcal{O}$, where $\mathcal{O}$ is the class of outer functions. 

As an application of our outer interpolation results, we determine 
$$\mathscr{R}(\K_u) : = \bigcap \big\{T_{\overline{\phi}} \K_u: \phi \in \mathcal{O} \cap H^{\infty}\big\},$$
the common range on a fixed model space. While $\mathscr{R}(H^2)$ is somehow ``large'', for example, it contains all functions that are analytic in a neighborhood of $\overline{\D}$,  $\mathscr{R}(\K_u)$ can be considerably smaller. In fact $\mathscr{R}(\K_u) = \{0\}$ for certain $u$ (Example \ref{nosmooth}). We describe $\mathscr{R}(\K_u)$  for any inner function (Theorem \ref{notusefule8shyshhH}) and when $u$ is an interpolating Blaschke product (zeros are an interpolating sequence), we give an alternate, and more tangible,  description involving our outer interpolating results (Theorem \ref{Hdecereuiusvnn77}). 

\section{Some notation}\label{notataions8}

Let us set our notation and review some well-known facts about the classes of analytic functions that appear in this paper. The books \cite{Duren, Garnett} are thorough references for the details and proofs. 
In this paper, $\D$ is the open unit disk $\{z \in \C: |z| < 1\}$, $\T$ the unit circle $\{z \in \C: |z| = 1\}$, and $dm = d \theta/2\pi$ is normalized Lebesgue measure on $\T$. 
The space $H^1$, the {\em Hardy space}, is the set of analytic functions $f$ on $\D$ for which 
$$\sup_{0 < r < 1} \int_{\T} |f(r \xi)| dm(\xi) < \infty.$$
Standard results  say that every $f \in H^1$ has a radial limit 
$$f(\xi): = \lim_{r \to 1^{-}} f(r \xi)$$ for almost every $\xi \in \T$ and 
$$\int_{\T} |f(\xi)| dm(\xi) = \sup_{0 < r < 1} \int_{\T} |f(r \xi)| dm(\xi).$$
As is the usual practice in Hardy spaces, we use the symbol $f$ to denote the boundary function on $\T$ as well as the analytic function on $\D$. 

We let $H^{\infty}$ be the bounded analytic functions on $\D$ and observe that $H^{\infty} \subset H^1$ and thus every $f \in H^{\infty}$ also has a radial boundary function. In fact, 
$$\sup_{z \in \D} |f(z)| = \mbox{ess-sup}_{\T} |f|.$$

If  $W$ is an extended real-valued integrable function on $\T$, 
\begin{equation}\label{outerfunctions88u}
\phi(z) = \exp\Big(\int_{\T} \frac{\xi + z}{\xi - z} W(\xi) dm(\xi)\Big), \quad z \in \D,
\end{equation}
is analytic on $\D$ and is called an {\em outer function}.
Observe  that 
$$
\log|\phi(z)|  = \int_{\T} \Re\Big(\frac{\xi + z}{\xi - z}\big) W(\xi)\,dm(\xi) 
=  \int_{\T} \frac{1 - |z|^2}{|\xi - z|^2} W(\xi)\,dm(\xi),$$
which is the Poisson integral of $W$. 
By some harmonic analysis \cite[p.~15]{Garnett}, 
$$\lim_{r \to 1^{-}} \log|\phi(r \zeta)| = \log |\phi(\zeta)| = W(\zeta)$$ for almost every $\zeta \in \T$. This process can be reversed and so given an extended real-valued $W \in L^1(\T)$, there is an outer $\phi$ with 
\begin{equation}\label{eeewwww}
|\phi(\zeta)| = e^{W(\zeta)}
\end{equation}
 almost everywhere (in terms of radial boundary values).

The outer functions belong to the {\em Smirnov class}
$$N^{+} = \big\{f/g: f \in H^{\infty}, g \in H^{\infty} \cap \mathcal{O}\big\}$$ and every $F \in N^{+}$ can be factored as $F = I_{F} O_{F}$, where $I_{F}$ is inner ($I_F \in H^{\infty}$ with unimodular boundary values almost everywhere on $\T$) and $O_{F}$ is outer. There are also the inclusions 
$H^{\infty} \subset H^1 \subset N^{+}$.

\section{Positive results}

We start off with examples of  bounded $\{w_n\}$ which can be interpolated by outer functions and explore the ones which can not in the next section. 

\begin{Theorem}\label{IT001}
Suppose $\{\lambda_n\} \subset \D$ is interpolating. For a bounded  $\{w_n\}$ with
$$\inf_{n \geq 1} |w_n| > 0,$$ there is a bounded outer function $\phi$ such that $\phi(\lambda_n) = w_n$ for all $n$. 
\end{Theorem}

The proof requires a few preliminaries. The first is a more detailed version of Carleson's result on interpolating sequences \cite[p.~268]{MR1669574}.

\begin{Theorem}[Carleson] \label{T:interpol}
For an interpolating $\Lambda = \{\lambda_n\}$ there is a constant $0<C(\Lambda)<1$ with the following properties. 
\begin{enumerate}[(i)]
\item For each  $\{w_n\}$ satisfying
$
|w_n| \leq C(\Lambda)$ for all $n \geq 1$,
there is an $\phi \in H^\infty$ such that $\|\phi\|_\infty \leq 1$ and
$\phi(\lambda_n)=w_n$ for all $n$.
\item There are bounded  $\{w_n\}$ with
$
|w_n| > C(\Lambda),
$
for at least one $n$,
such that $\|\phi\|_{\infty} > 1$ for any interpolating function $\phi \in H^{\infty}$.
\end{enumerate}
\end{Theorem}

The number $C(\Lambda)$ is called the {\em Carleson index} for $\Lambda$ and is related to $\delta(\Lambda)$ from \eqref{deltaaaa} by 
$$A \frac{\delta(\Lambda)}{\log\big(\dfrac{e}{\delta(\Lambda)}\big)} \leq C(\lambda) \leq \delta(\Lambda),$$
where $A$ is an absolute constant \cite[p.~268]{MR1669574}.  Our first step is to prove a special case of Theorem \ref{IT001}.

\begin{Lemma} \label{L:circle}
Suppose $\{w_n\} \subset \overline{D(a,r)} = \{z: |z - a| \leq r\}$. If
$r/|a| < C(\Lambda)
$
there is a bounded outer $\phi$ such that $\phi(\lambda_n) = w_n$ for all $n$. 
\end{Lemma}

The hypothesis $r/|a| < C(\Lambda)$ implies that  $r < |a|$ and so $\overline{D(a,r)} $ does not contain the origin.

\begin{proof}
Set 
$$
t_n :=  C(\Lambda) \frac{w_n-a}{r}.
$$
Then $|t_n| \leq C(\Lambda)$, and, by Theorem \ref{T:interpol}, there is $g \in H^\infty$ with $\|g\|_\infty \leq 1$ and such that $g(\lambda_n)=t_n$ for all $n$. Define
\[
\phi := \frac{r}{C(\Lambda)}g+a.
\]
Clearly $\phi \in H^{\infty}$ and 
\[
\phi(\lambda_n) = \frac{r}{C(\Lambda)}g(\lambda_n)+a = \frac{r}{C(\Lambda)}t_n+a = w_n.
\]
Moreover, with $a = |a|e^{i\alpha}$ and $z \in \D$, 
\begin{align*}
\Re(e^{-i\alpha}\phi(z)) & = \frac{r}{C(\Lambda)} \Re(e^{-i\alpha}g(z)) + |a| 
\geq \frac{r}{C(\Lambda)} \Big(\Re(e^{-i\alpha}g(z)) + 1 \Big) 
\end{align*}
which is positive. 
The condition
$\Re(e^{-i\alpha} \phi)>  0$ 
is sufficient to ensure that $e^{-i\alpha} \phi$ is an outer function \cite[p.~65]{Garnett}. Thus $\phi$ is outer as well.
\end{proof}

\begin{Remark}\label{detailreg0}
When $\{w_n\} \subset (0, \infty)$,  we can choose $\alpha = 0$ in the above proof and  thus choose the interpolating function to satisfy $\Re \phi > 0$. This detail will become important later on. 
\end{Remark}

\begin{Remark}
If one just wanted a nonvanishing interpolating function $\phi \in H^{\infty}$ in Theorem \ref{IT001}, one could take $\phi = e^{\psi}$, where $\psi \in H^{\infty}$ with $\psi(\lambda_n) = \log w_n$ (for a suitably defined logarithm). See \cite{MR2309193} for more on this. 
\end{Remark}

\begin{proof}[Proof of Theorem \ref{T:interpol}]
Fix $0<r<C(\Lambda)$ and consider the closed disk $\overline{D(1,r)}$. Since 
$
m \leq |w_n| \leq M$, $n \geq 1$,
for some positive constants $m$ and $M$, there is a positive integer $k$ such that
$
w_n^{1/k} \in \overline{D(1,r)}$ for all $n \geq 1$.
We use the main branch of logarithm to evaluate $w_n^{1/k}$. Therefore, by Lemma \ref{L:circle}, there is a bounded outer function $g$ such that
$
g(\lambda_n) = w_n^{1/k}$ for all $n$.
The function $\phi := g^k$ is bounded, outer, and 
$
\phi(\lambda_n) = w_n$ for all $n$.
\end{proof}


This next result says that for outer interpolation, we can always assume, for example, that the targets $w_n$ are positive. 

\begin{Proposition}\label{comp886bb}
Suppose $\{\lambda_n\}$ is interpolating and  $\{w_{n}\}$ and $\{w_{n}'\}$ are bounded with
$$0 < m \leq \big|\frac{w_{n}'}{w_n}\big| \leq M < \infty, \quad n \geq 1.$$
Then $\{w_n\}$ can be interpolated by an outer (bounded outer) function if and only if $\{w_{n}'\}$ can be interpolated by an outer  (bounded outer)  function. 
\end{Proposition}

\begin{proof}
 By Theorem \ref{IT001} there is a bounded outer $F$ such that 
$F(\lambda_n) = w_{n}'/w_{n}$ for all $n \geq  1$. 
If there is an outer (bounded outer)  $\phi$ such that  $\phi(\lambda_n) = w_n$ for all $n$, then the outer (bounded outer)  $\phi F$ performs the desired interpolation for $\{w_{n}'\}$. 
\end{proof}

\begin{Remark}
If $\phi$ is outer (bounded outer) then so is $\phi^{c}$ for any $c > 0$. Thus $\{w_{n}^{c}\}$ can be interpolated by an outer (bounded outer) function whenever $\{w_n\}$ can. 
\end{Remark}


\section{Negative results -- existence of an inner factor}

If $\{\lambda_n\}$ is interpolating we know that given any bounded  $\{w_n\}$ there is a $\phi \in H^{\infty}$ such that $\phi(\lambda_n) = w_n$. This next result says that under certain circumstances, any Smirnov interpolating function for $\{w_n\}$ must have an inner factor. 

\begin{Theorem} \label{T:interpol-outer-inner}
If $\{\lambda_n\}$ is interpolating and $\{w_n\} \subset \C \setminus \{0\}$ satisfies 
\[
\lim_{n \to \infty} (1-|\lambda_n|)\log|w_n| \ne 0,
\]
then any $\phi \in N^{+}$ satisfying $\phi(\lambda_n)  = w_n$ for all $n$ must have a non-trivial inner factor.
\end{Theorem}

\begin{Example} If
$$
w_n := e^{-\frac{1}{1-|\lambda_n|}}, \quad n \geq 1,
$$
then any interpolating $\phi \in N^{+}$ for $\{w_n\}$ is not outer. 
\end{Example}

This result will follow from the lemma below which is probably folklore but we include a proof for completeness. 

\begin{Lemma} \label{L:growth-outer}
If $\varphi$ is outer, then 
\[
\lim_{|z| \to 1^{-}} (1-|z|) \log|\varphi(z)| = 0.
\]
\end{Lemma}

\begin{proof}
Let $a > 1$ and 
$
E_a = \{\xi \in \T : |\varphi(\xi)|>a \}.
$
Then $\log|\phi| > 0$ on $E_a$ and an application of 
\begin{equation}\label{Pposequal1}
\int_{\T} \frac{1 - |z|^2}{|\xi - z|^2}\,dm(\xi) = 1, \quad z \in \D;
\end{equation}
and 
\begin{equation}\label{poissiner5}
\frac{1 - |z|^2}{|\xi - z|^2} \leq \frac{2}{1 - |z|}, \quad z \in \D,  \quad \xi \in \T.
\end{equation}
give us 
\begin{align*}
\log|\varphi(z)|  &= \int_{\T} \frac{1-|z|^2}{|z-\xi|^2} \log|\varphi(\xi)| dm(\xi)\\
&= \int_{E_a} \frac{1-|z|^2}{|z-\xi|^2} \log|\varphi(\xi)| dm(\xi) + \int_{\T\setminus E_a} \frac{1-|z|^2}{|z-\xi|^2} \log|\varphi(\xi)| dm(\xi)\\
&\leq \frac{2}{1-|z|} \int_{E_a}  \log|\varphi(\xi)| dm(\xi)
+ \log a \int_{\T\setminus E_a} \frac{1-|z|^2}{|z-\xi|^2} dm(\xi)\\
&\leq \frac{2}{1-|z|}\int_{E_a}  \log|\varphi(\xi)| dm(\xi) + \log a.
\end{align*}
Hence,
\[
(1-|z|)\log|\varphi(z)| \leq 2\int_{E_a}  \log|\varphi(\xi)| dm(\xi)+ (1-|z|)\log a, \quad z \in \D,
\]
which implies
\[
\varlimsup_{|z| \to 1^{-}}(1-|z|)\log|\varphi(z)| \leq 2\int_{E_a}  \log|\varphi(\xi)| dm(\xi).
\]
Now let $a \to +\infty$ and use the fact that $\log |\phi| \in L^1(\T)$ to deduce
\begin{equation}\label{E:limsup-phi}
\varlimsup_{|z| \to 1^{-}}(1-|z|)\log|\varphi(z)| \leq 0.
\end{equation}

Since $1/\varphi$ is also  outer , the above argument implies
\[
\varlimsup_{|z| \to 1^{-}}(1-|z|)\log|1/\varphi(z)| \leq 0,
\]
or equivalently
\begin{equation}\label{E:liminf-phi}
\varliminf_{|z| \to 1^{-}}(1-|z|)\log|\varphi(z)| \geq 0.
\end{equation}
The result now follows by comparing \eqref{E:limsup-phi} and \eqref{E:liminf-phi}.
\end{proof}

Let us comment here that when the hypothesis of Theorem \ref{T:interpol-outer-inner} is satisfied, the inner factor that appears in the interpolating function $\phi$ plays a significant role in the decay of $\phi$. 

\begin{Corollary}\label{yuhsjkfdmgsewdf}
Suppose $\{\lambda_n\}$ is interpolating and $\{w_n\} \subset \C \setminus \{0\}$ is a bounded and satisfies 
\[
\lim_{n \to \infty} (1-|\lambda_n|)\log|w_n| \ne 0.
\]
If $I_{\phi}$ is the inner factor for a $\phi \in N^{+}$ for which $\phi(\lambda_n) = w_n$ for all $n$, then 
$$\varliminf_{n \to \infty} |I_{\phi}(\lambda_n)| = 0.$$
\end{Corollary}

\begin{proof}
Let $\phi = F_{\phi} I_{\phi}$, where $F_{\phi}$ is outer and $I_{\phi}$ is inner. If $|I_{\phi}(\lambda_n)| \geq \delta > 0$ for all $n$, then 
$I_{\phi}(\lambda_n) = w_n/F_{\phi}(\lambda_n)$ satisfies the hypothesis of Theorem \ref{IT001} and so there is a bounded outer  $\psi$ with 
$\psi(\lambda_n) = I_{\phi}(\lambda_n)$ and so $F_{\phi} \psi$ is outer and interpolates $w_n$. This says that $w_n$ can be interpolated by an outer function -- which it can not. 
\end{proof}

\begin{Remark}
The above says that a subsequence of $\{\lambda_n\}$ must approach 
$$\Big\{\xi \in \T: \varliminf_{z \to \xi} |I_{\phi}(z)| = 0\Big\},$$
the boundary spectrum of the inner factor $I_{\phi}$. This set will consist of the accumulation of the zeros of the Blaschke factor of $I_{\phi}$ as well as the support of the singular measure associated with the singular inner inner factor of $I_{\phi}$ \cite[p.~152]{MR3526203}. 
\end{Remark}

\section{Negative results -- existence of a Blaschke product}

This next result says that under the right circumstances, any Smirnov interpolating must have a Blaschke factor. 

\begin{Theorem}\label{Blakshkefactor}
Suppose $\{\lambda_n\}$ is interpolating and $\{w_n\} \subset \C \setminus \{0\}$ is bounded and satisfies  
\[
\varlimsup_{n \to \infty} (1-|\lambda_n|)|\log|w_n|| = \infty.
\]
Then any $\phi \in N^{+}$ for which $\phi(\lambda_n) = w_n$ for all $n \geq 1$ must have a Blaschke factor. 
\end{Theorem}

This says, for example,  that for an interpolating $\{\lambda_n\}$ any $\phi \in H^{\infty}$ for which 
$$\phi(\lambda_n) = \exp\Big(-\frac{1}{(1 - |\lambda_n|)^{2}}\Big)$$ (and such $\phi$ exist by Carleson's theorem) must have a Blaschke factor. 

The proof of this theorem follows from the following lemma on zero-free Smirnov functions. Any zero-free $\phi \in N^{+}$ can be written as 
\begin{equation}\label{zerofreestructure}
\phi(z) = \exp\Big(\int_{\T} \frac{\xi + z}{\xi - z} W(\xi) dm(\xi)\Big)  \exp\Big(-\int_{\T} \frac{\xi + z}{\xi - z}d \mu(\xi)\Big),\end{equation}
where $W$ is a real-valued  integrable function and $\mu$ is a positive measure that is singular with respect to Lebesgue measure $m$.

\begin{Lemma}
If $\phi \in N^{+}$ and zero free, then 
$$\varlimsup_{|z| \to 1^{-}} (1 - |z|) \big|\log |\phi(z)|\big| < \infty.$$
\end{Lemma}

\begin{proof}
From \eqref{zerofreestructure} we have 
$$\log|\phi(z)| = \int_{\T}  \frac{1-|z|^2}{|z-\xi|^2} W(\xi) dm(\xi) - \int_{\T}  \frac{1-|z|^2}{|z-\xi|^2}  d \mu(\xi).$$
The proof of Lemma \ref{L:growth-outer} shows that 
$$\lim_{|z| \to 1^{-}} (1 - |z|) \int_{\T}  \frac{1-|z|^2}{|z-\xi|^2} W(\xi) dm(\xi) = 0.$$ From \eqref{poissiner5} we have
$$0 \leq ( 1- |z|)  \int_{\T}  \frac{1-|z|^2}{|z-\xi|^2}  d \mu(\xi) \leq 2 \int_{\T} d \mu =2  \mu(\T).$$
Combine these two facts to prove the result. 
\end{proof}

\section{A growth rate characterization}

 We know  from Corollary \ref{yuhsjkfdmgsewdf} that 
if $\{w_n\}$ can be interpolated by an outer  function, then 
$$\lim_{n \to \infty} (1- \lambda_n) \log |w_n| = 0.$$
What is the decay rate of $(1 - \lambda_n) \log |w_n|$? Here we focus our attention on the case when $\{\lambda_n\} \subset (0, 1)$. Though it does not play a role in our results, it is known \cite[p. ~156]{Duren} that $\{\lambda_n\} \subset (0, 1)$ is interpolating if and only if there is a $0 < c < 1$ such that 
$$(1 - \lambda_{n + 1}) \leq c (1 - \lambda_{n}), \quad n \geq 1.$$
Such sequences are called {\em exponential sequences}. Naively speaking, the following result says that the decay  rate of $(1- \lambda_n) \log |w_n|$ is controlled by an absolutely continuous function. The sharpness of this observation will be studied in Theorem \ref{growthragfetd6dx6}.

\begin{Theorem}\label{MTone}
Suppose $\{\lambda_n\} \subset (0, 1)$ is interpolating and $\{w_n\} \subset \C \setminus \{0\}$ is a bounded  with 
$$M := \sup_{n \geq 1} |w_n|.$$
Suppose there is an outer function $\phi$ for which $\phi(\lambda_n) = w_n$ for all $n$. 
Then there is a positive, decreasing, integrable function $h$ on $[0, 1]$ such that 
$$- (1 - \lambda_n) \log\Big|\frac{w_n}{M}\Big| \leq \int_{0}^{1 - \lambda_n} h(t) dt, \quad n \geq 1.$$
\end{Theorem}

To prepare for the proof, we require a few comments on the growth of the Poisson kernel. 
We begin with the following two normalizing assumptions on $h$: 
\begin{equation}\label{RC}
\mbox{$h$ is right continuous};
\end{equation}
\begin{equation}\label{hiszero}
\mbox{$h$ is extended to $[0, \pi]$ by setting $h(x) = 0$ for $x \in [1, \pi]$}.
\end{equation}
The right continuity can be assumed since $h$ is monotone and thus has at most a countable number of jumps. Since the behavior around the origin is our main concern, extending the definition of $h$ to $[0, \pi]$ is merely for aesthetic purposes when working with Poisson integrals below. 
Let 
$$P_{r}(t) = \frac{1 - r^2}{1 - 2 r \cos t + r^2}, \quad 0 \leq r < 1, \quad -\pi \leq t \leq \pi,$$
be the standard Poisson kernel. We wish to examine the function 
\begin{equation}\label{arrrrrr}
A_h(r) = (1 - r) \int_{-\pi}^{\pi} P_{r}(t) h(|t|) \frac{d t}{2 \pi}.
\end{equation}

For $r \in [0, 1)$ and $t \in [-\pi, \pi]$ we have 
\begin{equation}\label{99uuUU}
\frac{(1 - r)(1 - r^2)}{1 - 2 r \cos t + r^2} = \frac{(1 + r) (1 - r)^2}{(1 - r)^2 + 4 r \sin^{2}(t/2)}.
\end{equation}
This yields
\begin{align*}
A_h(r) & = \int_{-\pi}^{\pi} \frac{(1 - r)(1 - r^2)}{1 - 2 r \cos t + r^2} h(|t|) \frac{dt}{2 \pi}\\
& = 2 \int_{0}^{\pi}   \frac{(1 + r) (1 - r)^2}{(1 - r)^2 + 4 r \sin^{2}(t/2)} h(t) \frac{dt}{2 \pi}\\
& \geq  2 \int_{0}^{1 - r}   \frac{(1 + r) (1 - r)^2}{(1 - r)^2 + 4 r \sin^{2}(t/2)} h(t) \frac{dt}{2 \pi}
\end{align*}
Note that 
$$t^2 \sin^{2}(\tfrac{1}{2}) \leq \sin^{2}(\tfrac{t}{2}) \leq \tfrac{1}{4} t^2$$ and so 
\begin{align*}
\frac{1}{\pi} \int_{0}^{1 - r}   \frac{(1 + r) (1 - r)^2}{(1 - r)^2 + 4  \sin^{2}(t/2)} h(t) dt & \geq \frac{1}{ \pi} \int_{0}^{1 - r} \frac{(1 - r)^2}{(1 - r)^2 + t^2} h(t) dt\\
& \geq \frac{1}{ \pi} \int_{0}^{1 - r} \frac{(1 - r)^2}{(1 - r)^2 + (1 - r)^2} h(t) dt\\
& = \frac{1}{2 \pi} \int_{0}^{1 - r} h(t) dt.
\end{align*}
In summary,
\begin{equation}\label{lower}
A_h(r) \geq \frac{1}{2 \pi} \int_{0}^{1 - r} h(t) dt.
\end{equation}

To obtain an upper bound, note that $A(r)$ is equal to the sum of
$$
 \frac{1}{ \pi} \int_{0}^{1 - r} \frac{(1 - r)(1 - r^2)}{(1 - r)^2 + 4 r \sin^{2}(t/2)} h(t) dt$$ 
 and
$$\frac{1}{ \pi} \int_{1 - r}^{\pi} \frac{(1 - r)(1 - r^2)}{1 - 2 r \cos t + r^2} h(t) dt.$$
 For the first integral, observe  that
 \begin{align*}  
 \frac{1}{ \pi} \int_{0}^{1 - r}  \frac{(1 + r) (1 - r)^2}{(1 - r)^2 + 4 r \sin^{2}(t/2)} h(t) dt
 &  \leq  \frac{1}{ \pi} \int_{0}^{1- r}  \frac{(1 + r) (1 - r)^2}{(1 - r)^2 } h(t) dt\\
 & \leq \frac{2}{ \pi} \int_{0}^{1 - r} h(t)dt.
 \end{align*}
 For the second integral we recall from \eqref{hiszero} that $h(\pi) = 0$ and we use the second mean value theorem \cite{MR2960787} for integrals, along with the right continuity of $h$ from \eqref{RC},  to see that 
$$ \frac{1}{ \pi} \int_{1 - r}^{\pi} \frac{(1 - r)(1 - r^2)}{1 - 2 r \cos t + r^2} h(t) dt = \frac{1}{\pi} h(1 - r) \int_{1 - r}^{t_{0}}  \frac{(1 - r)(1 - r^2)}{1 - 2 r \cos t + r^2}dt$$ for some $t_{0} \in [1 - r, \pi]$. Notice that 
\begin{align*}
\frac{1}{ \pi} h(1 - r) \int_{1 - r}^{t_{0}}  \frac{(1 - r)(1 - r^2)}{1 - 2 r \cos t + r^2}dt & =  \frac{1}{\pi} (1 - r) h(1 - r)  \int_{1 - r}^{t_{0}} P_{r}(t) dt \\
& \leq  (1 - r) h(1 - r).
\end{align*}
In the last step note the use of  
$$\int_{1 - r}^{t_0} P_{r}(t) dt \leq \int_{0}^{\pi} P_{r}(t) dt = \pi.$$
We now use the fact that $h$ is decreasing to obtain
$$ \int_{0}^{1 - r} h(t) dt \geq  h(1 - r) \int_{0}^{1 - r} dt = (1 - r) h(1 - r).$$
Put this all together to get 
\begin{equation}\label{puy6vvVV}
A_h(r) \leq \tfrac{2 + \pi}{\pi} \int_{0}^{1 - r} h(t) dt.
\end{equation}
Thus, combining \eqref{lower}  and \eqref{puy6vvVV} we have the summary estimate 
\begin{equation}\label{summaruyarrrr}
\frac{1}{2 \pi} \int_{0}^{1 - r} h(t) dt \leq A_h(r) \leq \tfrac{2 + \pi}{\pi} \int_{0}^{1 - r} h(t) dt.
\end{equation}

An important tool for our next step is  the {\em symmetric decreasing rearrangement}. If $E$ is a measurable subset of $  [-\pi, \pi]$ let 
$$E^{*} = (-\tfrac{1}{2} |E|, \tfrac{1}{2}|E|)$$
be the interval centered about $0$ for which $|E| = |E^{*}|$, where $|\cdot|$ is Lebesgue measure on $[-\pi, \pi]$. For $f \in L^{1}[-\pi, \pi]$ with $f \geq 0$, define
$$f^{*}(x) = \int_{0}^{\infty} \chi_{\{f > t\}^{*}}(x) dt, \quad x \in [-\pi, \pi].$$
This function $f^{*}$ satisfies $f^{*}(x) = f^{*}(|x|)$ on $[- \pi, \pi]$ (i.e., symmetric), is non-increasing on $[0, \pi]$, and has the same integral as $f$. The important fact used here is the following \cite[Ch.~10]{MR0046395}. 

\begin{Lemma}[Hardy--Littlewood]\label{hhhllll}
For nonnegative and measurable $f, g$ we have 
$$\int_{-\pi}^{\pi} f(x) g(x)\,dx \leq \int_{-\pi}^{\pi} f^{*}(x) g^{*}(x)\,dx.$$
\end{Lemma}

If $f$ is positive and symmetric on $[-\pi, \pi]$ and $f$ is decreasing on $[0, \pi]$, then for each $t$ the set  $\{f > t\}$ is the interval 
$$\big(- \tfrac{1}{2} |\{f > t\}|,  \tfrac{1}{2}|\{f > t\}|\big).$$
In other words, $\{f > t\}^{*} = \{f > t\}$. 
By the layer cake representation of $f$ we have
\begin{align*}
f(x) & = \int_{0}^{\infty} \chi_{\{f > t\}}(x)dt\\
& = \int_{0}^{\infty} \chi_{\{f > t\}^{*}}(x) dt\\
& = f^{*}(x).
\end{align*}
Conclusion: If $f$ is positive, symmetric ($f(t) = f(|t|)$ on $[-\pi, \pi]$), and decreasing on $[0, \pi]$, then $f^{*} = f$ (almost everywhere).

\begin{proof}[Proof of Theorem \ref{MTone}]
If $\phi$ is outer and $\phi(\lambda_n) = w_n$ for all $n \geq 1$, then $\psi = \phi/M$ is outer and interpolates $w_n/M$. 
For $0 < r < 1$ we have 
\begin{align*}
- (1 - r) \log |\psi(r)| & = - (1 - r) \int_{-\pi}^{\pi} P_{r}(t) \log|\psi(e^{i t})| \frac{dt}{2 \pi}\\
& =  - (1 - r) \int_{-\pi}^{\pi} P_{r}(t) \max(0, \log|\psi(e^{i t})|) \frac{dt}{2 \pi}\\
& \quad -  (1 - r) \int_{-\pi}^{ \pi} P_{r}(t) \min(0, \log|\psi(e^{i t})|) \frac{dt}{2 \pi}\\
& \leq (1 - r) \int_{-\pi}^{\pi} P_{r}(t) k(t) \frac{dt}{2 \pi},
\end{align*}
where 
$k = -\min(0, \log |\psi|)$ is nonnegative and integrable on $[-\pi, \pi]$. 
 
 Apply the Hardy--Littlewood estimate (Lemma \ref{hhhllll})  to \eqref{arrrrrr} with $f = k$ and $g = P_{r}$ (which is already symmetric and so $g = g^{*}$ -- see the discussion above) to obtain the estimate 
 $$(1 - r) \int_{-\pi}^{\pi} P_{r}(t) k(t) \frac{dt}{2 \pi} \leq (1 - r) \int_{-\pi}^{\pi} P_{r}(t) k^{*}(t) \frac{dt}{2 \pi}.$$

By \eqref{summaruyarrrr} (note the use of the fact that $k^{*}(|t|) = k^{*}(t)$)
$$-(1 - r) \log|\psi(r)| \leq \frac{2 + \pi}{\pi} \int_{0}^{1 - r} k^{*}(t) dt.$$
Insert  $r = \lambda_n$ into the above inequality to complete the proof. 
\end{proof}

Next we improve Theorem \ref{MTone} with this sharpness result.




\begin{Theorem}\label{growthragfetd6dx6}
Suppose $\{\lambda_n\} \subset (0, 1)$ is interpolating and $h$ is a positive, decreasing, integrable function on $[0, 1]$. 
 If $\{w_n\} \subset \C \setminus \{0\}$ is bounded and satisfies  
$$-(1 - \lambda_n) \log |w_n| \asymp \int_{0}^{1 - \lambda_n} h(t) dt,$$ then there is a bounded outer function $\psi$ such that 
$$-(1 - \lambda_n) \log \psi(\lambda_n) \asymp \int_{0}^{1 - \lambda_n} h(t) dt.$$
\end{Theorem}

In the statement above, $A_{n} \asymp B_{n}$ means there are positive constants $c_1$ and $c_2$, independent of $n$, such that $c_1 A_{n} \leq B_n \leq c_2 A_n$ for all $n$. 

\begin{proof}
For $0 < r < 1$,  \eqref{summaruyarrrr} yields 
$$A_h(r) \asymp \int_{0}^{1 - r} h(t)dt.$$

If $\phi$ is the (bounded) outer function with 
\begin{equation}\label{67ygvvFGVFTYGFTYUHJ}
|\phi(e^{i t})| = e^{-h(|t|)}
\end{equation}
 for almost every $t \in [-\pi, \pi]$ (see \eqref{eeewwww}), then 
 \begin{align*}
 - (1 - r) \log |\phi(r)|  & = - (1 - r) \int_{-\pi}^{\pi} P_{r}(t) \log|\phi(e^{i t})| \frac{dt}{2 \pi}\\
 & = A_h(r)
  \asymp \int_{0}^{1 - r} h(t)dt. 
\end{align*}

With $r = \lambda_n$ we have 
$$ -(1 - \lambda_n) \log |\phi(\lambda_n)| \asymp \int_{0}^{1 - \lambda_n} h(t)dt.$$
Now apply Proposition \ref{comp886bb} to produce a bounded outer function $\psi$ with $\psi(\lambda_n) = |\phi(\lambda_n)|$. 
\end{proof}

\section{More delicate interpolation}

Given $h$ as in Theorem \ref{growthragfetd6dx6}, there is a bounded outer  $\phi$ such that $$\log \phi(\lambda_n) \asymp \frac{1}{1 - \lambda_n} \int_{0}^{1 - \lambda_n} h(t)\,dt, \quad n \geq 1.$$ Can we replace $\asymp$ with $=$ in the above? 
Equivalently, can we find an outer  (bounded outer) $\phi$ such that 
$$\phi(\lambda_n) = \exp\Big(-  \frac{1}{1 - \lambda_n} \int_{0}^{1 - \lambda_n} h(t) dt\Big), \quad n \geq 1?$$ 

We certainly can find 
$$d_n \in \Big[\frac{1}{2 \pi}, \frac{2 + \pi}{\pi}\Big]$$ such that 
$$\phi(\lambda_n)^{d_n} =  \exp\Big(- \frac{1}{1 - \lambda_n} \int_{0}^{1 - \lambda_n} h(t) dt\Big).$$
By Theorem \ref{IT001} and Remark \ref{detailreg0} there is a bounded outer $\psi$ with $\Re \psi > 0$ such that 
$\psi(\lambda_n) = 1/d_n$ for all $n$. 
The function 
$f = \phi^{\psi}$ is analytic on $\D$ with 
$$f(\lambda_n) =  \exp\Big(-  \frac{1}{1 - \lambda_n} \int_{0}^{1 - \lambda_n} h(t) dt\Big)$$ 
and thus performs the interpolation. 
But of course we need to check that $f$ is outer (bounded outer). 

Indeed this is something that needs checking since if $\phi$ and $\psi$ are outer,  $f = \phi^{\psi}$ need not be outer. In fact with $\phi = e$ (constant outer function) and 
$$\psi(z) = -\frac{1 + z}{1 - z},$$ then 
$$f = \phi^{\psi} = \exp\Big(-\frac{1 + z}{1 - z}\Big)$$ is inner! Here is our result concerning when $\phi^{\psi}$ is outer (bounded outer). 

\begin{Proposition} \label{L:ftog}
Let $\phi$ be outer and $\psi$ be bounded and outer. 
\begin{enumerate}
\item[(i)] If $\arg \phi(\xi) \in L^1(\T)$, then $f = \phi^{\psi}$ is outer. 
\item[(ii)] If $\Re \psi > 0$ and $\arg \phi(\xi) \in L^{\infty}(\T)$, then $f = \phi^{\psi}$ is outer and bounded. 
\end{enumerate}
\end{Proposition}

The proof of this proposition needs a few preliminaries. If $u \in L^1(\T)$ and $u \geq 0$, the {\em Herglotz integral }
\begin{equation}\label{bbherggg}
H_{u}(z) = \int_{\T} \frac{\xi + z}{\xi - z} u(\xi) dm(\xi)
\end{equation}
 is analytic on $\D$ and 
$$\Re H_{u}(z) = \int_{\T} \frac{1 - |z|^2}{|\xi - z|^2} u(\xi) dm(\xi) > 0, \quad z \in \D.$$
By a known result \cite[p.~65]{Garnett}, $H_{u}$ is outer. Recall from \S \ref{notataions8} the Hardy space $H^1$ and the Smirnov class $N^{+}$. 

\begin{Lemma}
For $f \in H^1$ there are $G_j \in N^{+}$ with $\Re G_j \geq 0$ on $\D$ for $j = 1, 2$ such that $f = G_1 - G_2$. 
\end{Lemma}

\begin{proof}
Functions in $H^1$ have radial boundary values almost everywhere on $\T$ and so let $u_{+}$ and $u_{-}$ be defined for almost every $\xi \in \T$ by 
$$u_{+} (\xi)= \max(\Re f(\xi), 0), \quad u_{-}(\xi) =  \max(-\Re f(\xi), 0).$$
Since $|\Re f(\xi)| \leq |f(\xi)|$ and $|f|$ is integrable on $\T$, we see that $u_{+}, u_{-}$ are nonnegative integrable functions. Furthermore, by the discussion above, $H_{u_{+}}$ and $H_{u_{-}}$ belong to $N^{+}$ and have positive real parts on $\D$. Finally,   
$H_{u_{+}} - H_{u_{-}}$ belongs to  $N^{+}$ and has the same real part as $f$ on $\T$. Thus, by the uniqueness of the harmonic conjugate, 
$f = H_{u_{+}} - H_{u_{-}} + i c$ for some $c \in \R$. This completes the proof. 
\end{proof}

\begin{Lemma}\label{h1outern8}
If $f \in H^1$, then $e^f$ is  outer. 
\end{Lemma}

\begin{proof}
By the previous lemma, 
$f = H_{u_{+}} - H_{u_{-}} + i c$ and so 
$$e^f = e^{i c} \frac{e^{-H_{u_{-}}}}{e^{-H_{u_{+}}}}.$$
From the formula for the Herglotz integral in \eqref{bbherggg} and the definition of outer from  \eqref{outerfunctions88u}, the functions $e^{-H_{u_{+}}}$ and $e^{-H_{u_{-}}}$ are outer. Thus, $e^f$ is also outer. 
\end{proof}

\begin{proof}[Proof of Proposition \ref{L:ftog}]
On $\T$ we have 
\begin{align*}
|\log \varphi| &\leq \big|\log|\varphi|\big| + |\arg \varphi| \\
&= \big|\log(|\varphi|/\|\varphi\|_\infty)+\log\|\varphi\|_\infty\big| + |\arg \varphi| \\
&\leq  \big|\log(|\varphi|/\|\varphi\|_\infty)\big|+\big|\log\|\varphi\|_\infty\big| + |\arg \varphi| \\
&=-\log(|\varphi|/\|\varphi\|_\infty)+\big|\log\|\varphi\|_\infty\big| + |\arg \varphi| \\
&\leq -\log|\varphi| + 2\log^+\|\varphi\|_\infty+ |\arg \phi|.
\end{align*}
From $\log |\phi| \in L^1(\T)$ and  $|\arg \phi | \in L^{1}(\T)$, follows $|\log \phi| \in L^1(\T)$. Since $\phi$ is outer, $\log \phi \in N^{+}$. A standard result \cite[p.~28]{Duren} of Smirnov implies $\log \phi \in H^{1}$. Therefore, $\psi \log \varphi \in H^1$. By the previous lemma, $f= \exp(\psi \log \varphi)$ is outer. This proves (i). 

If we assume that 
\[
|\Im \log \varphi| = |\arg \varphi| \leq M \quad \mbox{and} \quad \Re \psi \geq 0
\]
on $\T$,
we have 
\begin{align*}
|f| &= \exp( \Re \psi \, \log|\varphi| - \Im \psi \, \arg \varphi)\\
&\leq  \exp( \|\psi\|_\infty \log(1+\|\varphi\|_\infty) + M\|\psi\|_\infty).
\end{align*}
Thus, $f$ is a bounded outer function.  Note that  
$$\Re \psi \, \log|\varphi| \leq \|\psi\|_\infty \log(1+\|\varphi\|_\infty)$$ follows from the fact that $\Re \psi \geq 0$ on $\T$. This proves (ii).
\end{proof}

Let us use the results above to refine Theorem \ref{growthragfetd6dx6}. 

\begin{Theorem}\label{MTThree}
Suppose $h$ is a positive, decreasing, integrable function on $[0, 1]$.
Let $\{\lambda_n\} \subset (0, 1)$ be interpolating  and 
$\{w_n\} \subset \C \setminus \{0\}$ be bounded with 
$$-(1 - \lambda_n) \log |w_n| \asymp \int_{0}^{1 - \lambda_n} h(t)dt, \quad n \geq 1.$$
\begin{enumerate}
\item[(i)] If $h(|t|) \log^{+} h(|t|) \in L^{1}[-\pi, \pi]$ then there is an outer  $\phi$ such that 
$\phi(\lambda_n) = w_n$ for all $n$. 
\item[(ii)] If 
$$\operatorname{PV} \int_{-\pi}^{\pi} \cot\Big(\frac{\theta - t}{2}\Big) h(|t|) \frac{dt}{2 \pi}$$ is bounded on $[-\pi, \pi]$
then there is a bounded outer $\phi$ such that $\phi(\lambda_n) = w_n$ for all $n$.
\end{enumerate}
\end{Theorem}

\begin{proof}
From the discussion at the very beginning of this section, we can find bounded outer $\phi$ and $\psi$ such that $f = \phi^{\psi}$ satisfies $f(\lambda_n) = w_n$ for all $n$. We just need to check that $f$ is outer (bounded outer). 

By the proof of Theorem \ref{growthragfetd6dx6} and \eqref{67ygvvFGVFTYGFTYUHJ}, $\log |\phi(e^{i t})| = -h(|t|)$ and 
\begin{align*}
\phi(z) & = \exp\Big(\int_{\T} \frac{\xi + z}{\xi - z} \log|\phi(\xi)| dm\Big)\\
& = \exp\Big(\int_{\T} \Re\big( \frac{\xi + z}{\xi - z} \big) \log|\phi(\xi)|dm + i \int_{\T} \Im\big( \frac{\xi + z}{\xi - z} \big) \log|\phi(\xi)|dm\Big).
\end{align*}
From
$$\arg \phi(z) = i \int_{\T} \Im\big( \frac{\xi + z}{\xi - z} \big) \log|\phi(\xi)|dm, \quad z \in \D,$$ and standard theory involving the Hilbert transform on the circle we have 
$$\arg \phi(e^{i \theta}) = - \operatorname{PV} \int_{-\pi}^{\pi} \cot\Big(\frac{\theta - t}{2}\Big) h(|t|) \frac{dt}{2 \pi}.$$
A classical result of Zygmund \cite[Vol I, p. 254]{MR1963498} says that if the function  $h(|t|) \log^{+} h(|t|)$ belongs to $ L^{1}[-\pi, \pi]$ then $\arg \phi \in L^{1}(\T)$. An application of Proposition \ref{L:ftog} yields $f = \phi^{\psi}$ is outer. 

If the above Hilbert transform is bounded, another application of Proposition \ref{L:ftog}, along with the fact that we can always choose $\psi$ so that $\Re \psi > 0$ (Remark \ref{detailreg0}),  yields $f = \phi^{\psi}$ is bounded and outer. 
\end{proof}

\begin{Example}
If $\{\lambda_n\} \subset (0, 1)$ is interpolating, we know from Theorem \ref{T:interpol-outer-inner} that any $\phi \in N^{+}$ with
$$\phi(\lambda_n) = \exp\Big(-\frac{2}{1 - \lambda_n}\Big), \quad n \geq 1$$ must have an inner factor. In fact, the obvious guess at an analytic functions that interpolates this sequence is 
$$\phi(z) = \exp\Big(-\frac{2}{1 - z}\Big)$$ turns out to be a constant multiple of an inner function. Indeed, the singular inner function 
$$\exp\Big(-\frac{1 + z}{1 - z}\Big)$$ can be written as 
\begin{align*}
\exp\Big(-\frac{1 + z}{1 - z}\Big) & = \exp\Big(- \frac{2 - (1 - z)}{1 - z}\Big)\\
& = \exp\Big(-\frac{2}{1 - z}\Big) e.
\end{align*}
Thus $\phi$ is a constant multiple of a singular inner function.
\end{Example} 

\begin{Example}
Let 
$$w_{n} = \exp\Big(-\frac{1}{1 - \lambda_n} \frac{1}{(\log\frac{100}{1 - \lambda_n})^2}\Big).$$
Here 
$$h(t)  = \frac{2}{t  (\log(\frac{100}{t}))^3}, \quad 0 < t < 1,$$
is  $h$ is positive and decreasing  on $[0, 1]$ and $h(|t|) \log^{+} h(|t|)$ belongs to $L^{1}[-1, 1]$. Thus $\{w_n\}$ can be interpolated with an outer function. 
\end{Example}

\begin{Example}
Let
$$w_{n} = \exp\Big(-\frac{1}{(1 - \lambda_n)^{\alpha}}\Big),$$
where $0 <\alpha < 1$. In this case,
$$h(t) =  \frac{1 - \alpha}{t^{\alpha}}$$ is positive, decreasing, and  $h(|t|) \log^{+} h(|t|) \in L^1[-\pi, \pi]$. Thus,  by the previous theorem, $\{w_n\}$ can be interpolated by an outer function. In fact, one can take $\phi$ to be a bounded outer function. To see this, observe that $(1 - z)^{-\alpha} \in H^1$ and so 
$$\phi(z) = \exp\Big(-\frac{1}{(1 - z)^{\alpha}}\Big)$$ is outer (Lemma \ref{h1outern8}). Furthermore, 
\begin{align*}
\frac{1}{1 - e^{i \theta}} &=  \frac{e^{i \theta/2}}{e^{- i \theta/2} - e^{i \theta/2}}\\
& = \frac{e^{- i \theta/2}}{-2 i \sin(\theta/2)}\\
& = \frac{1}{2 \sin(\theta/2)} e^{i \frac{\pi - \theta}{2}}.
\end{align*}
Thus, 
$$|\phi(e^{i \theta})| \leq e^{-2^{-\alpha} \cos(\pi \alpha/2)}, \quad \theta \in [-\pi, \pi],$$
and so $\phi$ is outer and $\phi$ is bounded on $\T$. A result of Smirnov \cite[p.~28]{Duren} says that $\phi \in H^{\infty}$. 
If $0 < m \leq d_n \leq M  < \infty$,  one can also interpolate 
$$w_n = \exp\Big(- d_n\frac{1}{(1 - \lambda_n)^{\alpha}}\Big)$$ with an outer function. 
\end{Example}

\begin{Example}
If $\{\lambda_n\} \subset (0, 1)$ is interpolating and $\{d_n\}$ satisfies $0 < m \leq d_n \leq M  < \infty$ for all $n \geq 1$, one can appeal to Proposition \ref{L:ftog} directly to interpolate 
$$w_{n} = (1 - \lambda_n)^{d_n}$$ with a bounded outer function. Here $f = \phi^{\psi}$, where $\phi(z) = 1 - z$ (which clearly has bounded argument) and $\psi$ is the bounded outer function with $\Re \psi > 0$ and $\psi(\lambda_n) = d_n$ for all $n \geq 1$. 
\end{Example}

\section{Common range}\label{Crrrsdsdsd546u}

In this section, we present an application of our outer interpolation results. 
For $\phi \in H^{\infty}$ let $T_{\overline{\phi}}$ denote the co-analytic Toeplitz operator on the Hardy space $H^2$. By this we mean the operator $T_{\overline{\phi}}: H^2 \to H^2$ defined by  $T_{\overline{\phi}} f = P_{+}(\overline{\phi} f),$
where $P_{+}$ is the standard orthogonal projection of $L^2(\T)$ onto $H^2$. See \cite[Ch.~4]{MR3526203} for the basics of Toeplitz operators. 
 Let 
$$\mathscr{R}(H^2) := \bigcap \big\{ T_{\overline{\phi} }H^2: \phi \in H^{\infty} \setminus \{0\}\big\}$$
denote the {\em common range} of the (nonzero)  co-analytic Toeplitz operators on $H^2$. A well-known result is the following:

\begin{Theorem}[McCarthy \cite{MR1065054}]\label{098ueiorge32}
$$\mathscr{R}(H^2) = \{f \in H^{\infty}: \widehat{f}(n) = O(e^{-c_f \sqrt{n}})\}.$$
\end{Theorem}

The above decay on  the Fourier coefficients 
$\widehat{f}(n)$
shows that $\{n^{K} \widehat{f}(n)\}$ is absolutely summable for all $K \geq 0$ and so functions in $\mathscr{R}(H^2)$ must be infinitely differentiable on $\overline{\D}$. 

The Douglas factorization theorem \cite{MR203464} implies that $T_{\overline{\phi}} H^2 = T_{\overline{\phi_0}}H^2 $, where $\phi_{0}$ is the outer part of $\phi \in H^{\infty}$. Thus, 
$$\mathscr{R}(H^2) =  \bigcap\big\{ T_{\overline{\phi} }H^2: \phi \in H^{\infty} \cap \mathcal{O}\big\}.$$
Recall that $\mathcal{O}$ are the outer functions. Also important here is that $T_{\overline{\phi}}$ is injective whenever $\phi \in H^{\infty} \cap \mathcal{O}$.

What does this common range problem look like in model spaces? For an inner function $u$, the {\em model space} 
$\mathcal{K}_u := (u H^2)^{\perp}$ is  an invariant subspace for any co-analytic Toeplitz operator $T_{\overline{\phi}}$, $\phi \in H^{\infty}$. For this and other facts about model spaces used in this section, we refer the reader to \cite{MR3526203}.
For a fixed inner function $u$, what is 
$$\mathscr{R}(\K_u) :=  \bigcap\big\{ T_{\overline{\phi} } \mathcal{K}_u: \phi \in H^{\infty} \cap \mathcal{O}\big\}?$$

Since $T_{\overline{\phi}} \mathcal{K}_u \subset T_{\overline{\phi}} H^2$ we have 
$\mathscr{R}(\mathcal{K}_u) \subset \mathscr{R}(H^2)$ but the inclusion can be strict (see below). Furthermore, $\mathscr{R}(\K_u) \subset \K_u$ since $T_{\overline{\phi}} \K_u \subset \K_u$ for all bounded outer $\phi$.

\begin{Example}
If $u(z) = z^{N}$ then $\mathcal{K}_u = \mathscr{P}_{N - 1}$, the polynomials of degree at most $N - 1$. Since $\phi$ is outer,  $T_{\overline{\phi}}$ is injective and so $T_{\overline{\phi}} \mathscr{P}_{N - 1} = \mathscr{P}_{N - 1}.$ So in this case
$\mathscr{R}(\K_u) = \mathscr{P}_{N - 1}$.
\end{Example}

\begin{Example}\label{77tRRWO}
In a similar way, for a finite Blaschke product $u$ with distinct zeros $\lambda_1, \ldots, \lambda_n$ in $\D$, we have 
$\mathcal{K}_u = \bigvee\{k_{\lambda_1}, \ldots, k_{\lambda_n}\},$
where 
$k_{\lambda}(z) = \frac{1}{1 - \overline{\lambda} z}$ are the Cauchy kernels for $H^2$. It follows, using 
\begin{equation}\label{5533388UgggGT}
T_{\overline{\phi}} k_{\lambda_j} = \overline{\phi(\lambda_j)} k_{\lambda_{j}},
\end{equation} 
 and $\phi(\lambda_j) \not = 0$,  that 
$\mathscr{R}(\mathcal{K}_u) =  \bigvee\{k_{\lambda_1}, \ldots, k_{\lambda_n}\}.$
\end{Example}

What are some inhabitants of $\mathscr{R}(\mathcal{K}_u)$ when $u$ is not finite a Blaschke product? If $\lambda$ is a zero of $u$ then $k_{\lambda} \in \mathcal{K}_u$ and so $k_{\lambda} \in \mathscr{R}(\K_u)$ as argued in Example \ref{77tRRWO}. It is more difficult to identify other elements of $\mathscr{R}(\K_u)$.

\begin{Remark}
Since $\mathscr{R}(\K_u) \subset \K_u$, then $\mathscr{R}(\K_u)$ will inherit the properties of functions in $\K_u$. For example, if 
$$\Big\{\xi \in \T: \varliminf_{z \to \xi} |u(z)| = 0\Big\},$$
the boundary  spectrum of $u$,
 omits an an arc $I$ of $\T$, then function in $\K_u$ will have an analytic continuation across $I$. Hence functions $\mathscr{R}(\K_u)$ will also have this  property. 
\end{Remark}


\begin{Example}\label{nosmooth}
It is possible to produce a suitable singular inner function $u$ for which $\K_u$ contains no nonzero smooth functions \cite{MR2198372}. Since $\mathscr{R}(\mathcal{K}_u)$ is contained in the smooth functions  (Theorem \ref{098ueiorge32}), it follows that $\mathscr{R}(\K_u) = \{0\}$. 
\end{Example}



\begin{Theorem}\label{notusefule8shyshhH}
If $u$ is inner then
$$\mathscr{R}(\K_u) = \K_{u} \cap \mathscr{R}(H^2) = \{f \in \K_u: \widehat{f}(n) = O(e^{-c_f \sqrt{n}})\}.$$
\end{Theorem}

\begin{proof}
The containment $\subset$ is automatic. Now suppose $f \in \K_u \cap \mathscr{R}(H^2)$. Given any $\phi \in H^{\infty} \cap \mathcal{O}$ there is a $g_{\phi} \in H^2$ for which $f = T_{\overline{\phi}} g_{\phi}$. Since $f \in \K_u$ we have 
$\langle f, u h\rangle = 0$ for all $h \in H^2$. Using the fact that $T_{\overline{\phi}}^{*} = T_{\phi}$ (which is just multiplication by $\phi$), this implies
$$0 = \langle f, u h\rangle = \langle T_{\overline{\phi}} g_{\phi},  u h\rangle = \langle  g_{\phi},  \phi u h\rangle, \quad h \in H^2.$$
The function $\phi$ is outer and so $\{\phi h: h \in H^2\}$ is dense in $H^2$ (Beurling's theorem \cite[p.~114]{Duren}). Thus $\langle g_{\phi}, u k\rangle = 0$ for all $k \in H^2$ and so $g_{\phi} \in \K_u$. Thus, $f \in \mathscr{R}(\K_u)$. 
\end{proof}

Though Theorem \ref{notusefule8shyshhH} is a description of $\mathscr{R}(\K_u)$, it can be difficult to apply. Indeed, the precise contents of a model space are not always well understood and thus determining which of them have the right smoothness property can be quite challenging. In the next section we focus on a special class of inner functions $u$ where we better understand $\K_u$ as well as  $\mathscr{R}(\K_u)$. 

\section{Interpolating Blaschke products}

In Example \ref{77tRRWO} we computed $\mathscr{R}(\mathcal{K}_{B})$ when $B$ is a finite Blaschke product. In this section we extend our discussion to interpolating Blaschke products.  Let 
$$\kappa_{\lambda} : = \frac{k_{\lambda}}{\|k_{\lambda}\|} = \frac{\sqrt{1 - |\lambda|^2}}{1 - \overline{\lambda} z}, \quad \lambda \in \D,$$
denote the normalized Cauchy kernel for $H^2$. This next proposition is a well-known fact about model spaces \cite[p.~277]{MR3526203}.

\begin{Proposition}\label{tyuhghdfsgh}
 If $B$ is an interpolating Blaschke product with zeros $\{\lambda_{n}\}$, then $\{\kappa_{\lambda_n}\}$ is a Riesz basis for $\mathcal{K}_B$. Hence each $f \in \mathcal{K}_B$ has a unique representation as 
$
f = \sum_{n \geq 1} a_n \kappa_{\lambda_n},
$
where $\{a_n\} \in \ell^2$, that is, $\sum_{n \geq 1} |a_n|^2 < \infty$. Conversely, any such linear combination belongs to $\mathcal{K}_B$.
\end{Proposition}




We now obtain a more tangible description of $\mathscr{R}(\K_B)$ than the one in Theorem \ref{notusefule8shyshhH}. We start with the following lemma. 

\begin{Lemma} \label{L:separated-seq-2}
Let $\varphi$ be bounded and outer  and  $B$ be an interpolating Blaschke product with zeros $\{\lambda_n\}$. Then 
\begin{equation}\label{LLKK8898ui342}
T_{\overline{\phi}} \mathcal{K}_{B} = \left\{ \sum_{n = 1}^{\infty} b_n \kappa_{\lambda_n} : \sum_{n = 1}^{\infty} \frac{|b_n|^2}{|\varphi(\lambda_n)|^2} < \infty \right\}.
\end{equation}
\end{Lemma}

\begin{proof}
Suppose 
$$\sum_{n = 1}^{\infty} \frac{|b_n|^2}{|\phi(\lambda_n)|^2} < \infty.$$ Then 
$$f =\sum_{n = 1}^{\infty} \frac{b_n}{\overline{\phi(\lambda_n)}} \kappa_{\lambda_n} \in \K_B$$ (Proposition \ref{tyuhghdfsgh}) and, by \eqref{5533388UgggGT},  
$$T_{\overline{\phi}} f = \sum_{n = 1}^{\infty} \frac{b_n}{\overline{\phi(\lambda_n)}} \overline{\phi(\lambda_n)} \kappa_{\lambda_n} = \sum_{n = 1}^{\infty} b_n \kappa_{\lambda_n}.$$
Thus, $\sum_{n \geq 1} b_n \kappa_{\lambda_n} \in T_{\overline{\phi}} \K_B.$

Conversely, suppose  $g = T_{\overline{\phi}} f$ for some $f \in \K_B$. Then 
$f = \sum_{n \geq 1} a_n \kappa_{\lambda_n}$ for some unique  $\{a_n\} \in \ell^2$. From Proposition \ref{tyuhghdfsgh} 
$$g = T_{\overline{\phi}} f = \sum_{n = 1}^{\infty} a_{n} \overline{\phi(\lambda_n)} \kappa_{\lambda_n}$$ and so
$\{a_{n} \overline{\phi(\lambda_n)}\} = \{b_n\} \in \ell^2$. Thus, $\{a_{n}\} = \{b_n/\overline{\phi(\lambda_n)}\} \in \ell^2$. 
\end{proof}

Here is our description of $\mathscr{R}(\K_B)$. 

\begin{Theorem}\label{Hdecereuiusvnn77}
Suppose $B$ is an interpolating Blaschke product with zeros $\{\lambda_n\} \subset (0, 1)$. For $\{a_n\} \in \ell^2$ and 
$f = \sum_{n \geq 1} a_n \kappa_{\lambda_n}$, the following are equivalent:
\begin{enumerate}
\item[(i)] $f \in \mathscr{R}(\K_B)$; 
\item[(ii)]  The sum 
$$\sum_{n = 1}^{\infty} \frac{|a_n|^2}{|\phi(\lambda_n)|}$$
is finite for every bounded outer function $\phi$. 
\item[(iii)] The sum
$$\sum_{n = 1}^{\infty} |a_n|^2 \exp\Big(\frac{1}{1 - \lambda_n} \int_{0}^{1 - \lambda_n} h(t)\,dt\Big)$$
is finite 
for every positive, decreasing, integrable function $h$ on $[0, 1]$.
\item[(iv)] $$\sum_{n = 1}^{\infty} a_n \sqrt{1 - |\lambda_n|^2} \lambda_{n}^{N} = O(e^{-c_f \sqrt{N}}), \quad N \to \infty.$$
\end{enumerate}
\end{Theorem}

\begin{proof}
The proof of (i) $\iff$ (ii) follows from Lemma \ref{L:separated-seq-2}. The proof of (ii) $\iff$ (iii) follows from Theorem  \ref{growthragfetd6dx6}.  For the proof that (i) $\iff$ (iv), note that 
$$f^{(N)}(z) = N! \sum_{n = 1}^{\infty} a_n \sqrt{1 - |\lambda_n|^2} \frac{\lambda_{n}^{N}}{(1 - \lambda_n z)^{N + 1}}$$
and thus 
$$\widehat{f}(N) = \frac{f^{(N)}(0)}{N!} = \sum_{n = 1}^{\infty} a_n \sqrt{1 - |\lambda_n|^2} \lambda_{n}^{N}.$$
Now apply Theorem \ref{notusefule8shyshhH}. 
\end{proof}




To obtain a rich class of functions in $\mathscr{R}(\K_B)$, besides the obvious finite linear combinations of $\kappa_{\lambda_n}$, Lemma \ref{L:growth-outer} says that 
$$- (1 - \lambda_n) \log |\phi(\lambda_n)| \to 0, \quad \phi \in H^{\infty} \cap \mathcal{O}.$$
If $c > 0$ and  $\{a_n\}$ satisfies 
$$\sum_{n = 1}^{\infty} |a_n|^2 \exp\Big(\frac{c}{1 - \lambda_n}\Big) < \infty,$$ then  
$$\frac{1}{|\phi(\lambda_n)|} \leq \exp\Big(\frac{c}{1 - \lambda_n}\Big)$$
for sufficiently large enough $n$.
Thus, $f = \sum_{n \geq 1} a_n \kappa_{\lambda_n} \in \mathscr{R}(\K_B)$. In other words, 
$$\bigcup_{c > 0} \Big\{\sum_{n = 1}^{\infty} a_{n} \kappa_{\lambda_n}: \sum_{n = 1}^{\infty} |a_n|^2 \exp\big(\tfrac{c}{1 - \lambda_n}\big) < \infty\Big\} \subset \mathscr{R}_{B}.$$

\section{Correct definition of the common range?}

When $\phi \in H^{\infty} \setminus \{0\}$ it is easy to show that $T_{\overline{\phi}} H^2$ is dense in $H^2$. Thus 
$$\bigcap\big\{ T_{\overline{\phi}}H^2: \phi \in H^{\infty} \setminus \{0\}\big\},$$
 the common range of the non-zero co-analytic Toeplitz operators, is meaningful. It just so happens, through the Douglas factorization theorem mentioned earlier, that $T_{\overline{\phi}} H^2 = T_{\overline{\phi_0}} H^2$, where $\phi_0$ is the outer factor of $\phi$ and thus 
$$\bigcap\big\{ T_{\overline{\phi}}H^2: \phi \in H^{\infty} \setminus \{0\}\big\} = \bigcap\big\{ T_{\overline{\phi}}H^2: \phi \in H^{\infty} \cap \mathcal{O}\big\}.$$
It made sense to us to define the common range of the co-analytic Toeplitz operators in the model space $\K_u$ as 
$$\bigcap\big\{ T_{\overline{\phi}} \K_{u}: \phi \in H^{\infty} \cap \mathcal{O}\big\}.$$
Of course the intersection 
$$\bigcap\big\{T_{\overline{\phi}}\K_u: \phi \in H^{\infty} \setminus \{0\}\big\} = \{0\}$$ since 
$\ker T_{\overline{u}} = \K_u$ so there needs to be some further restriction on the intersection. 
One might wonder if the ``correct'' definition of the common range in the model space should be
$$\bigcap\{ T_{\overline{\phi}} \K_u: \phi \in \mathcal{F}\big\},$$
where $\mathcal{F}$ is the set of $\phi \in H^{\infty}$ such that $T_{\overline{\phi}} \K_u$ is dense in $\K_u$. One could make a case for this definition. However, the resulting common range may not be all that interesting. 

For example, if $B$ is an interpolating Blaschke product  with zeros $\{\lambda_n\}$,  the fact that $\{\kappa_{\lambda_n}\}$ is minimal, in fact uniformly minimal \cite[p.~277]{MR3526203}, shows that $T_{\overline{\phi}} \K_B$ is dense in $\K_B$ if and only if $\phi(\lambda_n) \not = 0$ for all $n$. From \eqref{LLKK8898ui342} we have 
$$T_{\overline{\phi}} \mathcal{K}_{B} = \left\{ \sum_{n = 1}^{\infty} b_n \kappa_{\lambda_n} : \sum_{n = 1}^{\infty} \frac{|b_n|^2}{|\varphi(\lambda_n)|^2} < \infty \right\}.$$
Thus if $\phi \in H^{\infty}$ interpolates the nonzero values of $b_n$ (which can be done via Carleson's theorem), the quantity 
$$ \sum_{n = 1}^{\infty} \frac{|b_n|^2}{|\varphi(\lambda_n)|^2} $$ will be infinite whenever there are an infinite number of nonzero $b_n$. Thus $\mathcal{F}$ will consist of the $\phi \in H^{\infty}$ such that $\phi(\lambda_n) \not = 0$ for all $n$ and 
$$\bigcap\{ T_{\overline{\phi}} \K_B: \phi \in \mathcal{F}\big\}$$ 
will just be the  finite linear combinations of the $\kappa_{\lambda_n}$.

\bibliographystyle{plain}

\bibliography{CRreferences}

\end{document}